\documentclass[a4j,11pt,leqno]{amsart}
\makeatletter
\def\@settitle{\begin{center}
		\baselineskip14\p@\relax
		\normalfont\LARGE\bfseries
		\@title
		\ifx\@subtitle\@empty\else
		\\[1ex] 
		
		\normalsize\mdseries\@subtitle
		\fi
		\ifx\@didication\@empty\else
		\\[2ex] 
		
		\large\mdseries\it\@dedication
		\fi
	\end{center}
}
\def\subtitle#1{\gdef\@subtitle{#1}}
\def\@subtitle{}
\def\dedication#1{\gdef\@dedication{#1}}
\def\@dedication{}
\makeatother
\makeatletter
\renewcommand{\section}{\@startsection
	{section}{1}{0mm}{5mm}{2mm}{\raggedright\bfseries}}

\@addtoreset{equation}{section}
\makeatother
\usepackage{etoolbox}
\newcommand{\zerodisplayskips}{ 
	\setlength{\abovedisplayskip}{0pt} 
	\setlength{\belowdisplayskip}{0pt} 
	\setlength{\abovedisplayshortskip}{0pt} 
	\setlength{\belowdisplayshortskip}{0pt}}
\appto{\normalsize}{\zerodisplayskips}
\appto{\small}{\zerodisplayskips}
\appto{\footnotesize}{\zerodisplayskips}
\usepackage{caption}
\usepackage{fix-cm}    
\makeatletter
\newcommand\HUGE{\@setfontsize\Huge{80}{80}}
\makeatother  
\usepackage[shortlabels]{enumitem}
\usepackage{amsfonts}
\usepackage{amsthm}
\usepackage{amssymb}
\usepackage{mathtools}
\usepackage{float}
\usepackage{graphicx}
\usepackage{caption}
\usepackage[labelformat=simple]{subcaption}

\usepackage{array} 
\usepackage[export]{adjustbox}
\usepackage{tabularx}
\usepackage{multirow}
\newcolumntype{P}[1]{>{\centering\arraybackslash}p{#1}}
\newcolumntype{M}[1]{>{\centering\arraybackslash}m{#1}}
\usepackage[useregional=numeric]{datetime2}

\usepackage{tikz}
\tikzset{ 
roundnode/.style={draw,shape=circle, fill, inner sep=1pt},
smallround/.style={draw,shape=circle, fill, inner sep=0.5pt},
pointnode/.style={draw,shape=circle,inner sep=0pt}
}
\usetikzlibrary{arrows}
\usetikzlibrary{arrows.meta}
\usetikzlibrary{calc}
\usetikzlibrary{positioning}

\DeclareMathOperator{\subdiv}{Subdiv}
\DeclareMathOperator{\Supp}{Supp}
\DeclareMathOperator{\trop}{trop}
\DeclareMathOperator{\val}{val}

\def\al{{\epsilon}}
\def\bal{{\bar\epsilon}}

\newtheorem{theorem}{Theorem}[section] 
\newtheorem{Theorem}[theorem]{Theorem}
\newtheorem{Lemma}[theorem]{Lemma}
\newtheorem{Corollary}[theorem]{Corollary}
\newtheorem{Proposition}[theorem]{Proposition}
\theoremstyle{definition}
\newtheorem{Definition}[theorem]{Definition}
\newtheorem{Example}[theorem]{Example}

\newtheorem{Remark}[theorem]{Remark}

\title{Note on smoothness condition on tropical elliptic curves	of symmetric truncated cubic forms}
\author{Rani Sasmita Tarmidi}
\address{Rani Sasmita Tarmidi:
	Department of Mathematics, 
	Graduate School of Science, 
	Osaka University, 
	Toyonaka, Osaka 560-0043, Japan}
\email{u644627d@ecs.osaka-u.ac.jp,
	ranitarmidi@yahoo.com
}
\date{}
\begin{document}
\begin{abstract}
In this work, we provide explicit conditions for the coefficients of a symmetric truncated cubic to give a smooth tropical curve. We also examine non-smooth cases corresponding to some specific subdivision types.
\end{abstract}
\maketitle
\markboth{R.S.Tarmidi}
{Note on smoothness condition on tropical elliptic curves	of symmetric truncated cubic forms}
\tableofcontents
\section{Introduction}
Let $\mathbb{K}$ be a field with a valuation $\val: \mathbb{K} \to \mathbb{Q} \cup \{\infty\}$. In this paper, we work with a symmetric truncated cubic
\begin{equation}\label{eq-our-truncated-f}
f(x,y) = c_{12} (xy^2 + x^2y) + c_{34} (x^2 + y^2) + c_5 xy + c_{67} (x + y) + c_8 \in \mathbb{K}[x,y]
\end{equation}
and determine when its associated tropical polynomial
\begin{align}\label{eq-our-truncated-trop-f}
\begin{split}
\trop(f)(X, Y) = \min (& v_{12} + X + 2Y, v_{12} + 2X + Y, v_{34} + 2X, v_{34} + 2Y, v_5 + X + Y, \\
& v_{67} + X, v_{67} + Y, v_8),
\end{split}
\end{align}
where $v_{k} = \val(c_{k})$ for $k \in \{12, 34, 5, 67, 8\}$, has a smooth tropical curve. The tropical curve $C(\trop(f))$ of tropical polynomial $\trop(f): \mathbb{R}^2 \to \mathbb{R}$ is the collection of its singular loci. The structure of a tropical curve is determined by the subdivision of the Newton polygon of $\trop(f)$. The smooth tropical curves are dual to unimodular subdivisions. Since the Newton polygon of $\trop(f)$ is symmetric around $y=x$ and truncated, there are five unimodular subdivisions.
\begin{theorem}\label{thm-smooth}
Let $f$ be the symmetric truncated cubic in \textup{(\ref{eq-our-truncated-f})}. Then the possible cycles appearing in the tropical curve of $\trop(f)$ are triangles, squares, pentagons, hexagons and heptagons. Each of these cycles occurs if and only if for $k \in \{12, 34, 5, 67, 8\}$, the coefficients $v_{k}$ satisfy inequalities listed in \textup{Table \ref{characterization_table}}.
\begin{table}[ht]
\centering
\begin{tabular}{rlr}
 & Cycle shape & Conditions of $v_{k}$ \\ \hline
\multirow{3}{3 em}{ (a) } & \multirow{3}{8 em}{ Triangle } & $-v_{34} + 2v_{67} - v_8 < 0$  \\ 
&& $v_{12} - v_5 - v_{67} + v_8 < 0$ \\
&& $-2v_{12} + 3v_5 - v_8 < 0$ \\ \hline 
\multirow{3}{3 em}{ (b) } & \multirow{3}{8 em}{ Square } & $-v_5 + 2 v_{67} - v_8 < 0$ \\ 
&& $-v_{12} + 2 v_5 - v_{67} < 0$\\
&& $v_{12} - v_{34} - v_5 + v_{67} < 0$ \\ \hline 
\multirow{3}{3 em}{ (c) } & \multirow{3}{8 em}{ Pentagon } & $v_5 - 2 v_{67} + v_8 < 0$ \\ 
&& $-v_{12} + v_5 + v_{67}- v_8 < 0$\\
&& $v_{12} - v_{34} - v_5 + v_{67} < 0$ \\ \hline 
\multirow{3}{3 em}{ (d) } & \multirow{3}{8 em}{ Hexagon } & $-v_5 + 2 v_{67} - v_8 < 0$ \\ 
&& $-v_{34} + v_5 < 0$\\
&& $-v_{12} + v_{34} + v_5 - v_{67} < 0$ \\ \hline 
\multirow{3}{3 em}{ (e) } & \multirow{3}{8 em}{ Heptagon } & $v_5 -2 v_{67}+ v_8 < 0$ \\ 
&& $-v_{34} + 2 v_{67} - v_8 < 0$ \\
&& $-v_{12} + v_{34} + v_5 - v_{67} < 0$ \\ \hline 
\end{tabular}\caption{Conditions of $v_{12}, v_{34}, v_5, v_{67}, v_8$ for all smooth tropical curves of $\trop(f)$.}
\label{characterization_table}
\end{table}
\end{theorem}
Meanwhile, the non-smooth tropical curves of $\trop(f)$ are the duals of non-unimodular subdivisions. We also investigate possible non-unimodular subdivisions and provide the conditions of $(v_{12}, v_{34}, v_5, v_{67}, v_8)$ for some selected subdivisions. See Theorem \ref{thm-non-smooth} below. The above results will be applied in another paper \cite{nak23} that studies the tropicalization of a certain two-parameter family of Edwards elliptic curves closely. See $\S$\ref{sect-edwards} for more details. \\
\indent The contents of this paper are organized as follows. In Section \ref{sect-prelim}, we provide a necessary overview of the general definitions pertaining to tropical curves. Moving on to Section \ref{sect-smooth}, we present the characterization of smooth tropical curves of our symmetric truncated cubic. In Section \ref{sect-non-smth}, we delve into a discussion on the non-smooth tropical curves associated with our cubic. Lastly, in Section \ref{sect-app}, we showcase the utilization of an integral unimodular transformation on $f$ and demonstrate the practical applications of our main findings.
\section{Preliminaries on tropical curves}\label{sect-prelim}
Let $\mathbb{K}$ be a field with a valuation $\val: \mathbb{K} \to \mathbb{Q} \cup \{\infty\}$. For a Laurent polynomial
$$
f(x,y) = \sum a_{ij}x^i y^j \in \mathbb{K}[x^{\pm 1}, y^{\pm 1}],
$$
we use the following definitions \cite{sturm}.
\begin{Definition}
The tropicalization of polynomial $f(x,y)$ is obtained by replacing each coefficient with its valuation and altering the operations $(\cdot, +)$ with operations $(+, \min)$. It is written as
$$
\trop(f)(X,Y) = \min (\val(a_{ij}) + i \cdot X + j \cdot Y).
$$
The tropical curve of $\trop(f)$ that is denoted by $C(\trop(f))$, is the collection of coordinates $(X,Y) \in \mathbb{R}^2$ where $\trop(f)$ is not differentiable. It forms a collection of vertices, bounded edges, and rays.
\end{Definition}
\begin{Remark}\label{max-min}
It is common to find other sources in literature that express the tropicalization of a polynomial by using operations $(+, \max)$. The tropical curve of the tropical polynomial in the form
$$
\trop(f^\prime) = \max(-\val(a_{ij}) + i \cdot X + j \cdot Y)
$$
and $C(\trop(f))$ are point-symmetry with respect to the origin $O$.
\end{Remark}
\begin{proof}
Let $(X,Y)$ be a point on $C(\trop(f))$. Then there exist $i_1j_1$ and $i_2j_2$ such that
$$
\val(a_{i_1j_1}) + i_1X + j_1Y = \val(a_{i_2j_2}) + i_2X + j_2Y
$$
and less than other terms $\val(a_{ij}) + iX + jY$. Thus we have 
$$
-\val(a_{i_1j_1}) + i_1(-X) + j_1(-Y) = -\val(a_{i_2j_2}) + i_2(-X) + j_2(-Y)
$$
and greater than other terms of $-\val(a_{ij}) + i(-X) + j(-Y)$. In other words, $(X,Y)$ is a point on $C(\trop(f))$ if and only if $(-X, -Y)$ is a point on $C(\trop(f^\prime))$. Thus, the tropical curves are point-symmetry with respect to the origin $O$.
\end{proof}
To determine the conditions of $\val(a_{ij})$ for a specific tropical curve, we will use its relationship with the subdivision of the following Newton polygon of $\trop(f)$. 
\begin{Definition}
The set
$$
\Supp_f = \{(i,j) \in \mathbb{Z}^2: a_{ij} \neq 0\}
$$
denotes the support of tropical polynomial $\trop(f)$. The Newton polygon of $\trop(f)$, that is denoted by $\Delta_f$, is the convex hull of $\Supp_f$.
\end{Definition}
The subdivision of $\Delta_f$ plays a crucial role in understanding the structure of the tropical curves. The definition of a regular subdivision depends on the valuations of non-zero coefficients $a_{ij}$ in the following manner.
\begin{Definition}
Let $v = (\val(a_{ij}) \vert a_{ij} \neq 0) \in \mathbb{R}^{\Supp_f}$. Furthermore, let $\Delta_f$ be the Newton polygon of $\trop(f)$ and $\overline{\Delta_f}$ be the convex hull of 
$$
\{ (i,j,\val(a_{ij})) \mid (i,j) \in \Supp_f \}  
\subseteq \mathbb{Z}^2 \times \mathbb{R}.
$$
The regular subdivision $\subdiv_v$ is the image of the corner edges of the upper part of $\overline{\Delta_f}$ under the projection to $\mathbb{Z}^2$ that subdivide $\Delta_f$ into smaller polygons. \\
\indent Each smaller polygon is called a cell. A cell is primitive when all of its lattice points are its vertices. It is unimodular if it is a triangle of area half. A subdivision is primitive (resp. unimodular) when all cells are primitive (resp. unimodular). \\
\indent The collection of vectors $v$ that yield the same regular subdivision forms a polyhedral cone in $\mathbb{R}^{\Supp_f}$. The collection of these cones defines the secondary fan of the Newton polygon $\Delta_f$.
\end{Definition}
Unimodular subdivisions are also the finest subdivisions. Thus, they correspond to the top-dimensional cones of the secondary fan. We observe that a unimodular cell or subdivision is always primitive, but the converse does not hold in general. Furthermore, the coarsest subdivision is the Newton polygon itself. The tropical curve $C(\trop(f))$ is dual to $\subdiv_v$ (cf. \cite{kmm09,vig09}). Therefore, there is a one-to-one correspondence between the edges of a regular subdivision and the edges of a tropical curve. This relation is our main tool to analyze the structure of tropical curves of $\trop(f)$.
\section{Smooth tropical curves characterization}\label{sect-smooth}
As mentioned earlier, let $f$ be the truncated cubic polynomial
\begin{align}\label{eq-our-trunc-f-sect3}
f(x,y) = c_{12} (xy^2 + x^2y) + c_{34} (x^2 + y^2) + c_5 xy + c_{67} (x + y) + c_8.
\end{align}
Its tropicalization is the piece-wise linear function 
\begin{align}\label{eq-our-trop-f-sect3}
\begin{split}
\trop(f)(X, Y) = \min (& v_{12} + X + 2Y, v_{12} + 2X + Y, v_{34} + 2X, v_{34} + 2Y, v_5 + X + Y, \\
& v_{67} + X, v_{67} + Y, v_8)
\end{split}
\end{align}
where $v_{k} = \val(c_{k})$ for $k \in \{12,34,5,67,8\}$. We assume $c_{k} \neq 0$ unless stated otherwise. The Newton polygon $\Delta_f$ of $\trop(f)$ in Figure \ref{newton_polygon_psi} is the convex hull of the set
$$
\Supp_f = \{(1,2), (2,1), (0,2), (2,0), (1,1), (0,1), (1,0)\}.
$$
\begin{figure}[H]
\centering
\resizebox{0.3\textwidth}{!}{
\begin{tikzpicture}
\node [roundnode, label={[label distance=-3.5pt]north east:\tiny $(2,1)$}] (2) at (2,1) {};
\node [roundnode, label={[label distance=-3.5pt]north:\tiny $(1,2)$}] (1) at (1,2) {};
\node [roundnode, label={[label distance=-5pt]south east:\tiny $(2,0)$}] (4) at (2,0) {};
\node [roundnode, label={[label distance=-5pt]north west:\tiny $(0,2)$}] (3) at (0,2) {};
\node [roundnode, label={[label distance=-5pt]north :\tiny $(1,1)$}] (5) at (1,1) {};
\node [roundnode, label={[label distance=-3.5pt]south:\tiny $(1,0)$}] (7) at (1,0) {};
\node [roundnode, label={[label distance=-3.5pt]north west:\tiny $(0,1)$}] (6) at (0,1) {};
\node [roundnode, label={[label distance=-5pt]south west:\tiny $(0,0)$}] (8) at (0,0) {};
\draw (1) -- (2) -- (4) -- (7) -- (8) -- (6) -- (3) -- (1);
\end{tikzpicture} 
}
\caption{Newton polygon $\Delta_f$.}
\label{newton_polygon_psi}
\end{figure}
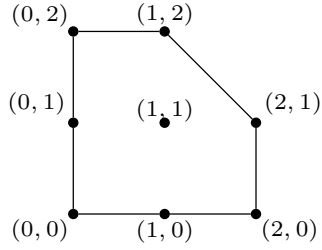
\begin{Definition}\label{def-smth-trp-crv}
A tropical curve is smooth when its dual subdivision is unimodular.
\end{Definition}
In this section, we will give all combinatorial possibilities of smooth tropical curve of $\trop(f)$. The Figure \ref{fig-sbdv-smth} below illustrates the possible unimodular subdivisions of $\Delta_f$.
\begin{figure}[ht]
\begin{subfigure}[b]{3cm}
\centering
\resizebox{3cm}{!}{ 
\begin{tikzpicture}
\node [pointnode, label={[label distance=-3.5pt]north:\tiny $1$}] (1) at (1,2) {};
\node [pointnode, label={[label distance=-5pt]north east:\tiny $2$}] (2) at (2,1) {};
\node [pointnode, label={[label distance=-5pt]south east:\tiny $3$}] (3) at (2,0) {};
\node [pointnode, label={[label distance=-5pt]north west:\tiny $4$}] (4) at (0,2) {};
\node [pointnode, label={[label distance=-3.5pt]north :\tiny $5$}] (5) at (1,1) {};
\node [pointnode, label={[label distance=-3.5pt]south:\tiny $6$}] (6) at (1,0) {};
\node [pointnode, label={[label distance=-5pt]north west:\tiny $7$}] (7) at (0,1) {};
\node [pointnode, label={[label distance=-5pt]south west:\tiny $8$}] (8) at (0,0) {};
\draw (1) -- (2) -- (3) -- (6) -- (8) -- (7) -- (4) -- (1);
\draw (5) -- (1);
\draw (5) -- (2);
\draw (5) -- (8);
\draw (1) -- (8);
\draw (1) -- (7);
\draw (2) -- (8);
\draw (2) -- (6);
\end{tikzpicture} 
}
\caption{}\label{fig:smooth_triangle}
\end{subfigure}
\begin{subfigure}[b]{3cm}
\centering
\resizebox{3cm}{!}{
\begin{tikzpicture}
\node [pointnode, label={[label distance=-3.5pt]north:\tiny $1$}] (1) at (1,2) {};
\node [pointnode, label={[label distance=-5pt]north east:\tiny $2$}] (2) at (2,1) {};
\node [pointnode, label={[label distance=-5pt]south east:\tiny $3$}] (3) at (2,0) {};
\node [pointnode, label={[label distance=-5pt]north west:\tiny $4$}] (4) at (0,2) {};
\node [pointnode, label={[label distance=-3.5pt]north :\tiny $5$}] (5) at (1,1) {};
\node [pointnode, label={[label distance=-3.5pt]south:\tiny $6$}] (6) at (1,0) {};
\node [pointnode, label={[label distance=-5pt]north west:\tiny $7$}] (7) at (0,1) {};
\node [pointnode, label={[label distance=-5pt]south west:\tiny $8$}] (8) at (0,0) {};
\draw (1) -- (2) -- (3) -- (6) -- (8) -- (7) -- (4) -- (1);
\draw (5) -- (1);
\draw (5) -- (2);
\draw (5) -- (7);
\draw (5) -- (6);
\draw (1) -- (7) -- (6) -- (2);
\end{tikzpicture} 
}
\caption{}\label{fig:smooth_square}
\end{subfigure}
\begin{subfigure}[b]{3cm}
\centering
\resizebox{3cm}{!}{
\begin{tikzpicture}
\node [pointnode, label={[label distance=-3.5pt]north:\tiny $1$}] (1) at (1,2) {};
\node [pointnode, label={[label distance=-5pt]north east:\tiny $2$}] (2) at (2,1) {};
\node [pointnode, label={[label distance=-5pt]south east:\tiny $3$}] (3) at (2,0) {};
\node [pointnode, label={[label distance=-5pt]north west:\tiny $4$}] (4) at (0,2) {};
\node [pointnode, label={[label distance=-3.5pt]north :\tiny $5$}] (5) at (1,1) {};
\node [pointnode, label={[label distance=-3.5pt]south:\tiny $6$}] (6) at (1,0) {};
\node [pointnode, label={[label distance=-5pt]north west:\tiny $7$}] (7) at (0,1) {};
\node [pointnode, label={[label distance=-5pt]south west:\tiny $8$}] (8) at (0,0) {};
\draw (1) -- (2) -- (3) -- (6) -- (8) -- (7) -- (4) -- (1);
\draw (5) -- (1);
\draw (5) -- (2);
\draw (5) -- (7);
\draw (5) -- (8);
\draw (5) -- (6);
\draw (1) -- (7);
\draw (2) -- (6);
\end{tikzpicture} 
}
\caption{}\label{fig:smooth_pentagon}
\end{subfigure} 
\begin{subfigure}[b]{3cm}
\centering
\resizebox{3cm}{!}{
\begin{tikzpicture}
\node [pointnode, label={[label distance=-3.5pt]north:\tiny $1$}] (1) at (1,2) {};
\node [pointnode, label={[label distance=-5pt]north east:\tiny $2$}] (2) at (2,1) {};
\node [pointnode, label={[label distance=-5pt]south east:\tiny $3$}] (3) at (2,0) {};
\node [pointnode, label={[label distance=-5pt]north west:\tiny $4$}] (4) at (0,2) {};
\node [pointnode, label={[label distance=-3.5pt]north :\tiny $5$}] (5) at (1,1) {};
\node [pointnode, label={[label distance=-3.5pt]south:\tiny $6$}] (6) at (1,0) {};
\node [pointnode, label={[label distance=-5pt]north west:\tiny $7$}] (7) at (0,1) {};
\node [pointnode, label={[label distance=-5pt]south west:\tiny $8$}] (8) at (0,0) {};
\draw (1) -- (2) -- (3) -- (6) -- (8) -- (7) -- (4) -- (1);
\draw (5) -- (1);
\draw (5) -- (2);
\draw (5) -- (4);
\draw (5) -- (7);
\draw (5) -- (6);
\draw (5) -- (3);
\draw (6) -- (7);
\end{tikzpicture} 
}
\caption{}\label{fig:smooth_hexagon}
\end{subfigure} 
\begin{subfigure}[b]{3cm}
\centering
\resizebox{3cm}{!}{
\begin{tikzpicture}
\node [pointnode, label={[label distance=-3.5pt]north:\tiny $1$}] (1) at (1,2) {};
\node [pointnode, label={[label distance=-5pt]north east:\tiny $2$}] (2) at (2,1) {};
\node [pointnode, label={[label distance=-5pt]south east:\tiny $3$}] (3) at (2,0) {};
\node [pointnode, label={[label distance=-5pt]north west:\tiny $4$}] (4) at (0,2) {};
\node [pointnode, label={[label distance=-3.5pt]north :\tiny $5$}] (5) at (1,1) {};
\node [pointnode, label={[label distance=-3.5pt]south:\tiny $6$}] (6) at (1,0) {};
\node [pointnode, label={[label distance=-5pt]north west:\tiny $7$}] (7) at (0,1) {};
\node [pointnode, label={[label distance=-5pt]south west:\tiny $8$}] (8) at (0,0) {};
\draw (1) -- (2) -- (3) -- (6) -- (8) -- (7) -- (4) -- (1);
\draw (5) -- (1);
\draw (5) -- (2);
\draw (5) -- (4);
\draw (5) -- (7);
\draw (5) -- (8);
\draw (5) -- (6);
\draw (5) -- (3);
\draw (5) -- (1);
\end{tikzpicture} 
}
\caption{}\label{fig:smooth_heptagon}
\end{subfigure} 
\caption{The unimodular subdivisions of $\Delta_f$.}
\label{fig-sbdv-smth}
\end{figure}
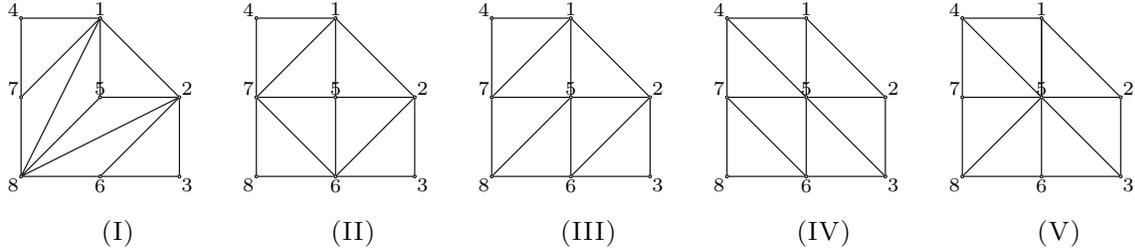
\begin{proof}[Proof of Theorem \ref{thm-smooth}]
We will prove the case (d). Tropical curve $C(\trop(f))$ has a hexagonal cycle when it is the dual of Figure \ref{fig:smooth_hexagon}. This subdivision can be written as
$$
\mathcal{S} = \{[1,2,5], [1,4,5], [2,3,5], [4,5,7], [3,5,6], [5,6,7], [6,7,8]\}.
$$
Each vertex of $C(\trop(f))$ corresponds to a cell of the subdivision as shown in Table \ref{tab-XYcoord-hexagon}.
\begin{table}[ht]
\begin{tabular}{lll}
Cell & $X$ coordinate & $Y$ coordinate \\ \hline
$[1,2,5]$ & $-v_{12} + v_5$ & $-v_{12} + v_5$ \\ \hline
$[1,4,5]$ & $-v_{12} + v_{34}$ & $-v_{12} + v_5$ \\ \hline
$[2,3,5]$ & $-v_{12} + v_5$ & $-v_{12} + v_{34}$ \\ \hline
$[4,5,7]$ & $-v_5 + v_{67}$ & $-v_{34} + v_{67}$ \\ \hline
$[3,5,6]$ & $-v_{34} + v_{67}$ & $-v_5 + v_{67}$ \\ \hline
$[5,6,7]$ & $-v_5 + v_{67}$ & $-v_5 + v_{67}$ \\ \hline
$[6,7,8]$ & $-v_{67} + v_8$ & $-v_{67} + v_8$
\end{tabular}
\caption{$(X,Y)$ coordinates of smooth $C(\trop(f))$ with a hexagonal cycle.}\label{tab-XYcoord-hexagon}
\end{table}
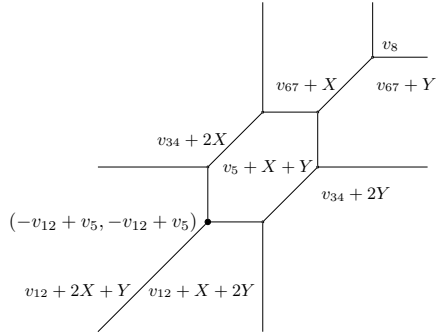
\begin{figure}[ht]
\centering
\resizebox{0.4\textwidth}{!}{ 
\begin{tikzpicture}
\node[roundnode, label = {west:\small $(-v_{12} + v_5, -v_{12} + v_5)$}] (1) at (0,0) {};
\node[pointnode] (2) at (1,0) {};
\node[pointnode] (3) at (2,1) {};
\node[pointnode] (4) at (2,2) {};
\node[pointnode] (5) at (1,2) {};
\node[pointnode] (6) at (0,1) {};
\draw (1) -- (2) -- (3) -- (4) -- (5) -- (6) -- (1);
\draw (1) -- ++(-2,-2);
\draw (2) -- ++(0,-2);
\draw (3) -- ++(2,0);
\draw (5) -- ++(0,2);
\draw (6) -- ++(-2,0);
\node[pointnode] (7) at (3,3) {};
\draw (4) -- (7);
\draw (7) -- ++(1,0);
\draw (7) -- ++(0,1);
\node at (-0.3,1) [right]{\footnotesize $v_{5} + X + Y$};
\node at (2.6,3.2) [right]{\footnotesize $v_{8}$};
\node at (2.5,2.5) [left]{\footnotesize $v_{67} + X$};
\node at (2.5,2.5) [right] {\footnotesize $v_{67} + Y$};
\node at (0.5,1.5) [left] {\footnotesize$v_{34} + 2X$};
\node at (1.5,0.5) [right] {\footnotesize$v_{34} + 2Y$};
\node at (-1.25,-1.25) [left] {\footnotesize $v_{12} + 2X + Y$};
\node at (1,-1.25) [left] {\footnotesize $v_{12} + X + 2Y$};
\end{tikzpicture} 
}
\caption{Tropical curve $C(\trop(f))$ with a hexagonal cycle.}
\label{tropical_curve_hexagon}
\end{figure}
\\ \noindent Vertex $(-v_{12} + v_5, -v_{12} + v_5)$ in Figure \ref{tropical_curve_hexagon} corresponds to cell $[1,2,5]$ means this vertex is the solution of the system of linear equations
$$
v_{12} + X + 2Y = v_{12} + 2X + Y = v_5 + X + Y
$$
which are the $1^{\text{st}}, 2^{\text{nd}}, 5^{\text{th}}$ terms of $\trop(f)(X, Y)$. Furthermore, 
\begin{align*}
& \trop(f) (-v_{12} + v_5, -v_{12} + v_5) \\
=& \min(-2v_{12} + 3v_{5}, -2v_{12} + 3v_{5}, v_{34} - 2v_{12} + 2v_{5}, v_{34} - 2v_{12} + 2v_{5}, \\
& -2v_{12} + 3v_{5}, v_{67} - v_{12} + v_{5}, v_{67} - v_{12} + v_{5}, v_{8}) \\
=& -2v_{12} + 3v_{5},
\end{align*}
which is the value of the $1^{\text{st}}, 2^{\text{nd}}, 5^{\text{th}}$ terms of $\trop(f) (-v_{12} + v_5, -v_{12} + v_5)$. Cell $[1,2,5]$ implies that $-2v_{12} + 3v_{5}$ is less than the other terms of $\trop(f) (-v_{12} + v_5, -v_{12} + v_5)$. It gives inequalities
\begin{equation}\label{125}
-2v_{12} + 3 v_{5} - v_{8} < 0 \qquad  
v_{5} - v_{34} < 0 \qquad  
-v_{12} + 2 v_{5} - v_{67} < 0.
\end{equation}
\noindent After applying the same procedure to the other six cells of $\mathcal{S}$, we have the regular subdivision $\mathcal{S}$ occurs if and only if inequalities 
\begin{center}
\begin{tabular}{rrr}
$  v_{5} - v_{34} < 0 $ & 
$  - v_{12} + 2v_{5} - v_{67} < 0 $ &
$  - 2v_{12} + 3v_{5} - v_{8} < 0 $ \\
$  - v_{12} + v_{34} + v_{5} - v_{67} < 0 $ & 
$  - 2v_{12} + v_{34} + 2v_{5} - v_{8} < 0 $ & 
$  - v_{5} + 2v_{67} - v_{8} < 0 $ \\
$  -v_{12} + 3v_{67} - 2v_{8} < 0 $ & 
$  -v_{34} + 2v_{67} - v_{8} < 0 $ &
\end{tabular}
\end{center}
hold. The above inequalities are equivalent to polyhedron
\begin{align}\label{ineq-smth-hexagon}
\begin{split}
-v_{5} + 2v_{67} - v_{8} < 0 \qquad  
- v_{34} + v_{5} < 0 \qquad  
- v_{12} + v_{34} + v_{5} - v_{67} < 0.
\end{split}
\end{align}
Thus, the tropical curve $C(\trop(f))$ is smooth with a hexagonal cycle if and only if (\ref{ineq-smth-hexagon}) holds. The same arguments hold for the cases of other cycles.
\end{proof}
\section{Non-smooth tropical curves}\label{sect-non-smth}
A tropical curve $C(\trop(f))$ is non-smooth when its dual is a non-unimodular subdivision. These subdivisions do not correspond to the top-dimensional cones of the secondary fan of $\Delta_f$. Thus, some or all of the inequalities in the conditions of $(v_{12}, v_{34}, v_5, v_{67}, v_8)$ will be non-strict inequalities. 
\begin{table}[ht]
\centering
\begin{tabular}{l|l|l}
1 cell &  
\resizebox{0.07\textwidth}{!}{
\begin{tikzpicture}
\node [pointnode] (2) at (2,1) {};
\node [pointnode] (1) at (1,2) {};
\node [pointnode] (4) at (2,0) {};
\node [pointnode] (3) at (0,2) {};
\node [pointnode] (5) at (1,1) {};
\node [pointnode] (7) at (1,0) {};
\node [pointnode] (6) at (0,1) {};
\node [pointnode] (8) at (0,0) {};
\draw (1) -- (2) -- (4) -- (7) -- (8) -- (6) -- (3) -- (1);
\end{tikzpicture}
}& \\ 
2 cells&  
\resizebox{0.07\textwidth}{!}{
\begin{tikzpicture}
\node [pointnode] (2) at (2,1) {};
\node [pointnode] (1) at (1,2) {};
\node [pointnode] (4) at (2,0) {};
\node [pointnode] (3) at (0,2) {};
\node [pointnode] (5) at (1,1) {};
\node [pointnode] (7) at (1,0) {};
\node [pointnode] (6) at (0,1) {};
\node [pointnode] (8) at (0,0) {};
\draw (1) -- (2) -- (4) -- (7) -- (8) -- (6) -- (3) -- (1);
\draw (6) -- (7);
\end{tikzpicture}
}
\resizebox{0.07\textwidth}{!}{
\begin{tikzpicture}
\node [pointnode] (2) at (2,1) {};
\node [pointnode] (1) at (1,2) {};
\node [pointnode] (4) at (2,0) {};
\node [pointnode] (3) at (0,2) {};
\node [pointnode] (5) at (1,1) {};
\node [pointnode] (7) at (1,0) {};
\node [pointnode] (6) at (0,1) {};
\node [pointnode] (8) at (0,0) {};
\draw (1) -- (2) -- (4) -- (7) -- (8) -- (6) -- (3) -- (1);
\draw (3) -- (4);
\end{tikzpicture}
} &\\ 
3 cells &  
\resizebox{0.07\textwidth}{!}{
\begin{tikzpicture}
\node [pointnode] (2) at (2,1) {};
\node [pointnode] (1) at (1,2) {};
\node [pointnode] (4) at (2,0) {};
\node [pointnode] (3) at (0,2) {};
\node [pointnode] (5) at (1,1) {};
\node [pointnode] (7) at (1,0) {};
\node [pointnode] (6) at (0,1) {};
\node [pointnode] (8) at (0,0) {};
\draw (1) -- (2) -- (4) -- (7) -- (8) -- (6) -- (3) -- (1);
\draw (6) -- (7);
\draw (3) -- (4);
\end{tikzpicture}
}
\resizebox{0.07\textwidth}{!}{
\begin{tikzpicture}
\node [pointnode] (2) at (2,1) {};
\node [pointnode] (1) at (1,2) {};
\node [pointnode] (4) at (2,0) {};
\node [pointnode] (3) at (0,2) {};
\node [pointnode] (5) at (1,1) {};
\node [pointnode] (7) at (1,0) {};
\node [pointnode] (6) at (0,1) {};
\node [pointnode] (8) at (0,0) {};
\draw (1) -- (2) -- (4) -- (7) -- (8) -- (6) -- (3) -- (1);
\draw (5) -- (8);
\draw (5) -- (1);
\draw (5) -- (2);
\end{tikzpicture}
}
\resizebox{0.07\textwidth}{!}{
\begin{tikzpicture}
\node [pointnode] (2) at (2,1) {};
\node [pointnode] (1) at (1,2) {};
\node [pointnode] (4) at (2,0) {};
\node [pointnode] (3) at (0,2) {};
\node [pointnode] (5) at (1,1) {};
\node [pointnode] (7) at (1,0) {};
\node [pointnode] (6) at (0,1) {};
\node [pointnode] (8) at (0,0) {};
\draw (1) -- (2) -- (4) -- (7) -- (8) -- (6) -- (3) -- (1);
\draw (1) -- (8);
\draw (2) -- (8);
\end{tikzpicture}
}
\resizebox{0.07\textwidth}{!}{
\begin{tikzpicture}
\node [pointnode] (2) at (2,1) {};
\node [pointnode] (1) at (1,2) {};
\node [pointnode] (4) at (2,0) {};
\node [pointnode] (3) at (0,2) {};
\node [pointnode] (5) at (1,1) {};
\node [pointnode] (7) at (1,0) {};
\node [pointnode] (6) at (0,1) {};
\node [pointnode] (8) at (0,0) {};
\draw (1) -- (2) -- (4) -- (7) -- (8) -- (6) -- (3) -- (1);
\draw (1) -- (6);
\draw (2) -- (7);
\end{tikzpicture}
} &
\resizebox{0.07\textwidth}{!}{
\begin{tikzpicture}
\node [pointnode] (2) at (2,1) {};
\node [pointnode] (1) at (1,2) {};
\node [pointnode] (4) at (2,0) {};
\node [pointnode] (3) at (0,2) {};
\node [pointnode] (5) at (1,1) {};
\node [pointnode] (7) at (1,0) {};
\node [pointnode] (6) at (0,1) {};
\node [pointnode] (8) at (0,0) {};
\draw (1) -- (2) -- (4) -- (7) -- (8) -- (6) -- (3) -- (1);
\draw (3) -- (4);
\draw (5) -- (8);
\end{tikzpicture}
} \\ 
4 cells &  
\resizebox{0.07\textwidth}{!}{
\begin{tikzpicture}
\node [pointnode] (2) at (2,1) {};
\node [pointnode] (1) at (1,2) {};
\node [pointnode] (4) at (2,0) {};
\node [pointnode] (3) at (0,2) {};
\node [pointnode] (5) at (1,1) {};
\node [pointnode] (7) at (1,0) {};
\node [pointnode] (6) at (0,1) {};
\node [pointnode] (8) at (0,0) {};
\draw (1) -- (2) -- (4) -- (7) -- (8) -- (6) -- (3) -- (1);
\draw (1) -- (7);
\draw (2) -- (6);
\end{tikzpicture}
}
\resizebox{0.07\textwidth}{!}{
\begin{tikzpicture}
\node [pointnode] (2) at (2,1) {};
\node [pointnode] (1) at (1,2) {};
\node [pointnode] (4) at (2,0) {};
\node [pointnode] (3) at (0,2) {};
\node [pointnode] (5) at (1,1) {};
\node [pointnode] (7) at (1,0) {};
\node [pointnode] (6) at (0,1) {};
\node [pointnode] (8) at (0,0) {};
\draw (1) -- (2) -- (4) -- (7) -- (8) -- (6) -- (3) -- (1);
\draw (1) -- (6) -- (7) -- (2);
\end{tikzpicture}
} & 
\resizebox{0.07\textwidth}{!}{
\begin{tikzpicture}
\node [pointnode] (2) at (2,1) {};
\node [pointnode] (1) at (1,2) {};
\node [pointnode] (4) at (2,0) {};
\node [pointnode] (3) at (0,2) {};
\node [pointnode] (5) at (1,1) {};
\node [pointnode] (7) at (1,0) {};
\node [pointnode] (6) at (0,1) {};
\node [pointnode] (8) at (0,0) {};
\draw (1) -- (2) -- (4) -- (7) -- (8) -- (6) -- (3) -- (1);
\draw (3) -- (4);
\draw (5) -- (6);
\draw (5) -- (7);
\end{tikzpicture}
}
\resizebox{0.07\textwidth}{!}{
\begin{tikzpicture}
\node [pointnode] (2) at (2,1) {};
\node [pointnode] (1) at (1,2) {};
\node [pointnode] (4) at (2,0) {};
\node [pointnode] (3) at (0,2) {};
\node [pointnode] (5) at (1,1) {};
\node [pointnode] (7) at (1,0) {};
\node [pointnode] (6) at (0,1) {};
\node [pointnode] (8) at (0,0) {};
\draw (1) -- (2) -- (4) -- (7) -- (8) -- (6) -- (3) -- (1);
\draw (3) -- (4);
\draw (5) -- (1);
\draw (5) -- (2);
\end{tikzpicture}
}\\ 
5 cells &  
\resizebox{0.07\textwidth}{!}{
\begin{tikzpicture}
\node [pointnode] (2) at (2,1) {};
\node [pointnode] (1) at (1,2) {};
\node [pointnode] (4) at (2,0) {};
\node [pointnode] (3) at (0,2) {};
\node [pointnode] (5) at (1,1) {};
\node [pointnode] (7) at (1,0) {};
\node [pointnode] (6) at (0,1) {};
\node [pointnode] (8) at (0,0) {};
\draw (1) -- (2) -- (4) -- (7) -- (8) -- (6) -- (3) -- (1);
\draw (1) -- (7);
\draw (2) -- (6);
\draw (5) -- (8);
\end{tikzpicture}
}
\resizebox{0.07\textwidth}{!}{
\begin{tikzpicture}
\node [pointnode] (2) at (2,1) {};
\node [pointnode] (1) at (1,2) {};
\node [pointnode] (4) at (2,0) {};
\node [pointnode] (3) at (0,2) {};
\node [pointnode] (5) at (1,1) {};
\node [pointnode] (7) at (1,0) {};
\node [pointnode] (6) at (0,1) {};
\node [pointnode] (8) at (0,0) {};
\draw (1) -- (2) -- (4) -- (7) -- (8) -- (6) -- (3) -- (1);
\draw (1) -- (7);
\draw (2) -- (6);
\draw (6) -- (7);
\end{tikzpicture}
}
\resizebox{0.07\textwidth}{!}{
\begin{tikzpicture}
\node [pointnode] (2) at (2,1) {};
\node [pointnode] (1) at (1,2) {};
\node [pointnode] (4) at (2,0) {};
\node [pointnode] (3) at (0,2) {};
\node [pointnode] (5) at (1,1) {};
\node [pointnode] (7) at (1,0) {};
\node [pointnode] (6) at (0,1) {};
\node [pointnode] (8) at (0,0) {};
\draw (1) -- (2) -- (4) -- (7) -- (8) -- (6) -- (3) -- (1);
\draw (3) -- (4);
\draw (5) -- (8);
\draw (1) -- (5) -- (2);
\end{tikzpicture}
}
\resizebox{0.07\textwidth}{!}{
\begin{tikzpicture}
\node [pointnode] (2) at (2,1) {};
\node [pointnode] (1) at (1,2) {};
\node [pointnode] (4) at (2,0) {};
\node [pointnode] (3) at (0,2) {};
\node [pointnode] (5) at (1,1) {};
\node [pointnode] (7) at (1,0) {};
\node [pointnode] (6) at (0,1) {};
\node [pointnode] (8) at (0,0) {};
\draw (1) -- (2) -- (4) -- (7) -- (8) -- (6) -- (3) -- (1);
\draw (1) -- (6);
\draw (2) -- (7);
\draw (5) -- (8);
\draw (1) -- (5) -- (2);
\end{tikzpicture}
}
\resizebox{0.07\textwidth}{!}{
\begin{tikzpicture}
\node [pointnode] (2) at (2,1) {};
\node [pointnode] (1) at (1,2) {};
\node [pointnode] (4) at (2,0) {};
\node [pointnode] (3) at (0,2) {};
\node [pointnode] (5) at (1,1) {};
\node [pointnode] (7) at (1,0) {};
\node [pointnode] (6) at (0,1) {};
\node [pointnode] (8) at (0,0) {};
\draw (1) -- (2) -- (4) -- (7) -- (8) -- (6) -- (3) -- (1);
\draw (5) -- (8);
\draw (1) -- (5) -- (2);
\draw (1) -- (8) -- (2);
\end{tikzpicture}
}
\resizebox{0.07\textwidth}{!}{
\begin{tikzpicture}
\node [pointnode] (2) at (2,1) {};
\node [pointnode] (1) at (1,2) {};
\node [pointnode] (4) at (2,0) {};
\node [pointnode] (3) at (0,2) {};
\node [pointnode] (5) at (1,1) {};
\node [pointnode] (7) at (1,0) {};
\node [pointnode] (6) at (0,1) {};
\node [pointnode] (8) at (0,0) {};
\draw (1) -- (2) -- (4) -- (7) -- (8) -- (6) -- (3) -- (1);
\draw (1) -- (8) -- (2);
\draw (1) -- (6);
\draw (2) -- (7);
\end{tikzpicture}
} & 
\resizebox{0.07\textwidth}{!}{
\begin{tikzpicture}
\node [pointnode] (2) at (2,1) {};
\node [pointnode] (1) at (1,2) {};
\node [pointnode] (4) at (2,0) {};
\node [pointnode] (3) at (0,2) {};
\node [pointnode] (5) at (1,1) {};
\node [pointnode] (7) at (1,0) {};
\node [pointnode] (6) at (0,1) {};
\node [pointnode] (8) at (0,0) {};
\draw (1) -- (2) -- (4) -- (7) -- (8) -- (6) -- (3) -- (1);
\draw (3) -- (4);
\draw (5) -- (6) -- (7) -- (5);
\end{tikzpicture}
}
\resizebox{0.07\textwidth}{!}{
\begin{tikzpicture}
\node [pointnode] (2) at (2,1) {};
\node [pointnode] (1) at (1,2) {};
\node [pointnode] (4) at (2,0) {};
\node [pointnode] (3) at (0,2) {};
\node [pointnode] (5) at (1,1) {};
\node [pointnode] (7) at (1,0) {};
\node [pointnode] (6) at (0,1) {};
\node [pointnode] (8) at (0,0) {};
\draw (1) -- (2) -- (4) -- (7) -- (8) -- (6) -- (3) -- (1);
\draw (3) -- (4);
\draw (5) -- (6);
\draw (5) -- (7);
\draw (5) -- (8);
\end{tikzpicture}
}\\ 
6 cells &  
\resizebox{0.07\textwidth}{!}{
\begin{tikzpicture}
\node [pointnode] (2) at (2,1) {};
\node [pointnode] (1) at (1,2) {};
\node [pointnode] (4) at (2,0) {};
\node [pointnode] (3) at (0,2) {};
\node [pointnode] (5) at (1,1) {};
\node [pointnode] (7) at (1,0) {};
\node [pointnode] (6) at (0,1) {};
\node [pointnode] (8) at (0,0) {};
\draw (1) -- (2) -- (4) -- (7) -- (8) -- (6) -- (3) -- (1);
\draw (1) -- (7);
\draw (2) -- (6);
\draw (1) -- (6);
\draw (2) -- (7);
\end{tikzpicture}
}
\resizebox{0.07\textwidth}{!}{
\begin{tikzpicture}
\node [pointnode] (2) at (2,1) {};
\node [pointnode] (1) at (1,2) {};
\node [pointnode] (4) at (2,0) {};
\node [pointnode] (3) at (0,2) {};
\node [pointnode] (5) at (1,1) {};
\node [pointnode] (7) at (1,0) {};
\node [pointnode] (6) at (0,1) {};
\node [pointnode] (8) at (0,0) {};
\draw (1) -- (2) -- (4) -- (7) -- (8) -- (6) -- (3) -- (1);
\draw (1) -- (7);
\draw (2) -- (6);
\draw (3) -- (4);
\end{tikzpicture}
} &
\end{tabular}
\caption{The non-unimodular subdivisions of $\Delta_f$.}
\label{table-sbdv-nonsmth}
\end{table}
Table \ref{table-sbdv-nonsmth} presents the subdivisions of $\Delta_f$ into polygons, some of its cells have areas greater than half. However, it is important to note that the corresponding tropical curves $C(\trop(f))$ associated with the subdivisions listed in the right column do not occur in practice.
\begin{Proposition}
Let $\trop(f)$ be as defined in \textup{(\ref{eq-our-trop-f-sect3})} and $\Delta_f$ be its Newton polygon. Then, the five subdivisions on the right column of \textup{Table \ref{table-sbdv-nonsmth}} never occur as the regular subdivisions of $\Delta_f$ for any $v = (v_{12}, v_{34}, v_5, v_{67}, v_8)$.
\end{Proposition}
\begin{proof}
The proof can be accomplished by examining the shape of the subdivision. Let us assume that one of the five subdivisions is viable. In doing so, we observe that the interior point $(1,1)$ forms a vertex of the Newton polygon $\Delta_f$. However, it is evident that its dual cannot form a closed cycle in a tropical curve.
\end{proof}
\begin{Theorem}\label{thm-non-smooth}
The \textup{Table \ref{characterization_table2}} below shows the necessary and sufficient conditions for valuations $(v_{12}, v_{34}, v_5, v_{67}, v_8)$ to give specific non-smooth tropical curve $C(\trop(f))$ listed in the left column.
\begin{table}[ht]
\centering
\begin{tabular}{rllr}
&Non-smooth tropical curves & Subdivisions & Conditions of $v_{ij}$ \\ \hline
\multirow{3}{1em}{(a)} & \multirow{3}{17 em}{ Four vertices with no cycle } &  
\multirow{3}{2cm}{
\resizebox{0.07\textwidth}{!}{
\begin{tikzpicture}
\node [pointnode] (2) at (2,1) {};
\node [pointnode] (1) at (1,2) {};
\node [pointnode] (4) at (2,0) {};
\node [pointnode] (3) at (0,2) {};
\node [pointnode] (5) at (1,1) {};
\node [pointnode] (7) at (1,0) {};
\node [pointnode] (6) at (0,1) {};
\node [pointnode] (8) at (0,0) {};
\draw (1) -- (2) -- (4) -- (7) -- (8) -- (6) -- (3) -- (1);
\draw (1) -- (6) -- (7) -- (2);
\end{tikzpicture}
}
}
&  $v_{12} - 2v_{34} + v_{67} < 0$\\
&&& $-v_{12} + 3v_{67} - 2v_8 \leq 0$ \\
&&& $v_{12} - 2v_5 + v_{67} \leq 0$ \\ \hline
\multirow{3}{1em}{(b)} & \multirow{3}{17 em}{ A square cycle with two bounded edges } &  
\multirow{3}{2cm}{
\resizebox{0.07\textwidth}{!}{
\begin{tikzpicture}
\node [pointnode] (2) at (2,1) {};
\node [pointnode] (1) at (1,2) {};
\node [pointnode] (4) at (2,0) {};
\node [pointnode] (3) at (0,2) {};
\node [pointnode] (5) at (1,1) {};
\node [pointnode] (7) at (1,0) {};
\node [pointnode] (6) at (0,1) {};
\node [pointnode] (8) at (0,0) {};
\draw (1) -- (2) -- (4) -- (7) -- (8) -- (6) -- (3) -- (1);
\draw (1) -- (6) -- (2) -- (7);
\draw (1) -- (7);
\end{tikzpicture}
}
}
& $-v_{12} + 3v_{67} - 2v_8 < 0$ \\
&&& $v_{12} - v_{34} - v_{67} + v_8 < 0$ \\
&&& $v_5 - 2v_{67} + v_8 = 0$ \\ \hline
\multirow{3}{1em}{(c)} & \multirow{3}{18 em}{A pentagon cycle with seven rays} &  
\multirow{3}{2cm}{
\resizebox{0.07\textwidth}{!}{
\begin{tikzpicture}
\node [pointnode] (2) at (2,1) {};
\node [pointnode] (1) at (1,2) {};
\node [pointnode] (4) at (2,0) {};
\node [pointnode] (3) at (0,2) {};
\node [pointnode] (5) at (1,1) {};
\node [pointnode] (7) at (1,0) {};
\node [pointnode] (6) at (0,1) {};
\node [pointnode] (8) at (0,0) {};
\draw (1) -- (2) -- (4) -- (7) -- (8) -- (6) -- (3) -- (1);
\draw (1) -- (7);
\draw (2) -- (6);
\draw (5) -- (8);
\end{tikzpicture}
}
}
& $-v_{34} + 2v_{67} - v_8 < 0$ \\
&&& $v_{12} - v_{34} - v_{67} + v_8 < 0$ \\
&&& $-v_{12} + v_{34} + v_5 - v_{67} = 0$ \\ \hline
\multirow{3}{1em}{(d)} & \multirow{3}{17 em}{ Five vertices with no cycle} &  
\multirow{3}{2cm}{
\resizebox{0.07\textwidth}{!}{
\begin{tikzpicture}
\node [pointnode] (2) at (2,1) {};
\node [pointnode] (1) at (1,2) {};
\node [pointnode] (4) at (2,0) {};
\node [pointnode] (3) at (0,2) {};
\node [pointnode] (5) at (1,1) {};
\node [pointnode] (7) at (1,0) {};
\node [pointnode] (6) at (0,1) {};
\node [pointnode] (8) at (0,0) {};
\draw (1) -- (2) -- (4) -- (7) -- (8) -- (6) -- (3) -- (1);
\draw (6) -- (1) -- (8) -- (2) -- (7);
\end{tikzpicture}
}
}
& $-v_{34} + 2v_{67} - v_8 < 0$ \\
&&& $v_{12} - 3v_{67} + 2v_8 < 0$ \\
&&& $2v_{12} - 3v_5 + v_8 \leq 0$ \\ \hline
\multirow{3}{1em}{(e)} & \multirow{3}{17 em}{ A triangle cycle with seven rays } &  
\multirow{3}{2cm}{
\resizebox{0.07\textwidth}{!}{
\begin{tikzpicture}
\node [pointnode] (2) at (2,1) {};
\node [pointnode] (1) at (1,2) {};
\node [pointnode] (4) at (2,0) {};
\node [pointnode] (3) at (0,2) {};
\node [pointnode] (5) at (1,1) {};
\node [pointnode] (7) at (1,0) {};
\node [pointnode] (6) at (0,1) {};
\node [pointnode] (8) at (0,0) {};
\draw (1) -- (2) -- (4) -- (7) -- (8) -- (6) -- (3) -- (1);
\draw (1) -- (5) -- (2);
\draw (1) -- (6);
\draw (5) -- (8);
\draw (2) -- (7);
\end{tikzpicture}
}
}
& $v_{12} - 3v_{67} + 2v_8 < 0$ \\
&&& $-v_{34} + 2v_{67} - v_8 < 0$ \\
&&& $-v_{12} + v_5 + v_{67} - v_8 = 0$ \\ \hline
\multirow{3}{1em}{(f)} & \multirow{3}{17 em}{ A square cycle with no bounded edge } &  
\multirow{3}{2cm}{
\resizebox{0.07\textwidth}{!}{
\begin{tikzpicture}
\node [pointnode] (2) at (2,1) {};
\node [pointnode] (1) at (1,2) {};
\node [pointnode] (4) at (2,0) {};
\node [pointnode] (3) at (0,2) {};
\node [pointnode] (5) at (1,1) {};
\node [pointnode] (7) at (1,0) {};
\node [pointnode] (6) at (0,1) {};
\node [pointnode] (8) at (0,0) {};
\draw (1) -- (2) -- (4) -- (7) -- (8) -- (6) -- (3) -- (1);
\draw (1) -- (7);
\draw (2) -- (6);
\end{tikzpicture}
}
}
& $v_{12} - 2v_{34} + v_{67} < 0$ \\
&&& $v_{12} - v_{34} - v_{67} + v_8 = 0$ \\
&&& $v_{12} - v_{34} - v_5 + v_{67}  = 0$ \\ \hline
\multirow{3}{1em}{(g)} & \multirow{3}{17 em}{ A square cycle with one bounded edge } &  
\multirow{3}{2cm}{
\resizebox{0.07\textwidth}{!}{
\begin{tikzpicture}
\node [pointnode] (2) at (2,1) {};
\node [pointnode] (1) at (1,2) {};
\node [pointnode] (4) at (2,0) {};
\node [pointnode] (3) at (0,2) {};
\node [pointnode] (5) at (1,1) {};
\node [pointnode] (7) at (1,0) {};
\node [pointnode] (6) at (0,1) {};
\node [pointnode] (8) at (0,0) {};
\draw (1) -- (2) -- (4) -- (7) -- (8) -- (6) -- (3) -- (1);
\draw (1) -- (7);
\draw (2) -- (6);
\draw (6) -- (7);
\end{tikzpicture}
}
}
& $v_{12} - 2v_{34} + v_{67} < 0$ \\
&&& $-v_{12} + v_{34} + v_{67} - v_8 <0$ \\ 
&&& $-v_{12} + v_{34} + v_5 - v_{67} = 0$ \\ \hline
\multirow{3}{1em}{(h)} & \multirow{3}{17 em}{A trivalent pentagon cycle} &  
\multirow{3}{2cm}{
\resizebox{0.07\textwidth}{!}{
\begin{tikzpicture}
\node [pointnode] (2) at (2,1) {};
\node [pointnode] (1) at (1,2) {};
\node [pointnode] (4) at (2,0) {};
\node [pointnode] (3) at (0,2) {};
\node [pointnode] (5) at (1,1) {};
\node [pointnode] (7) at (1,0) {};
\node [pointnode] (6) at (0,1) {};
\node [pointnode] (8) at (0,0) {};
\draw (1) -- (2) -- (4) -- (7) -- (8) -- (6) -- (3) -- (1);
\draw (3) -- (4);
\draw (1) -- (5) -- (2);
\draw (5) -- (8);
\end{tikzpicture}
}
}
& $-v_{34} + v_5 < 0$  \\
&&& $-2v_{12} + v_{34} + 2v_5 - v_8 < 0$ \\
&&& $v_{34} - 2v_{67} + v_8 \leq 0$  \\ \hline
\multirow{3}{1em}{(i)} & \multirow{3}{17 em}{No bounded edge} &  
\multirow{3}{2cm}{
\resizebox{0.07\textwidth}{!}{
\begin{tikzpicture}
\node [pointnode] (2) at (2,1) {};
\node [pointnode] (1) at (1,2) {};
\node [pointnode] (4) at (2,0) {};
\node [pointnode] (3) at (0,2) {};
\node [pointnode] (5) at (1,1) {};
\node [pointnode] (7) at (1,0) {};
\node [pointnode] (6) at (0,1) {};
\node [pointnode] (8) at (0,0) {};
\draw (1) -- (2) -- (4) -- (7) -- (8) -- (6) -- (3) -- (1);
\end{tikzpicture}
}
}
& $u_{8} + 2u_{12} - 3u_{34} = 0$ \\
&&& $-u_{12} + 2u_{34} - u_{67} \leq 0$ \\
&&& $-u_{5} + u_{34} \leq 0$ \\ \hline
\multirow{3}{1em}{(j)} & \multirow{3}{17 em}{One bounded edge with seven rays} &  
\multirow{3}{2cm}{
\resizebox{0.07\textwidth}{!}{
\begin{tikzpicture}
\node [pointnode] (2) at (2,1) {};
\node [pointnode] (1) at (1,2) {};
\node [pointnode] (4) at (2,0) {};
\node [pointnode] (3) at (0,2) {};
\node [pointnode] (5) at (1,1) {};
\node [pointnode] (7) at (1,0) {};
\node [pointnode] (6) at (0,1) {};
\node [pointnode] (8) at (0,0) {};
\draw (1) -- (2) -- (4) -- (7) -- (8) -- (6) -- (3) -- (1);
\draw (6) -- (7);
\end{tikzpicture}
}
}
& $-u_{8} - u_{34} + 2u_{67} < 0$ \\
&&& $-u_{5} + u_{34} \leq 0$ \\
&&& $u_{12} - 2u_{34} + u_{67} = 0$ \\ \hline
\end{tabular}\caption{Some non-smooth tropical curves $\trop(f)$.}\label{characterization_table2}
\end{table}
\end{Theorem}
\begin{proof}[Proof of Theorem \ref{thm-non-smooth}]
We will prove case (h) in detail, the non-smooth tropical curve with a trivalent pentagonal cycle. The collection of vectors $v$ corresponding to it and the collection of vectors $v$ yielding the subdivision in Figure \ref{fig-some-sbdv-cubic-non-smth} are coincide.
\begin{figure}[ht]
\begin{subfigure}[b]{7cm}
\centering
\resizebox{3cm}{!}{ 
\begin{tikzpicture}
\node [pointnode, label={[label distance=-3.5pt]north:\tiny $1$}] (1) at (1,2) {};
\node [pointnode, label={[label distance=-5pt]north east:\tiny $2$}] (2) at (2,1) {};
\node [pointnode, label={[label distance=-5pt]south east:\tiny $3$}] (3) at (2,0) {};
\node [pointnode, label={[label distance=-5pt]north west:\tiny $4$}] (4) at (0,2) {};
\node [pointnode, label={[label distance=-3.5pt]north :\tiny $5$}] (5) at (1,1) {};
\node [pointnode, label={[label distance=-3.5pt]south:\tiny $6$}] (6) at (1,0) {};
\node [pointnode, label={[label distance=-5pt]north west:\tiny $7$}] (7) at (0,1) {};
\node [pointnode, label={[label distance=-5pt]south west:\tiny $8$}] (8) at (0,0) {};
\draw (1) -- (2) -- (3) -- (6) -- (8) -- (7) -- (4) -- (1);
\draw (5) -- (1);
\draw (5) -- (2);
\draw (5) -- (4);
\draw (5) -- (8);
\draw (5) -- (3);
\end{tikzpicture} 
} 
\caption{Subdivision $\mathcal{S}$.}
\label{fig-some-sbdv-cubic-non-smth}
\end{subfigure}
\begin{subfigure}[b]{7cm}
\centering
\resizebox{5cm}{!}{ 
\begin{tikzpicture}
\node [pointnode] (1) at (-1,-3) {};
\node [pointnode] (2) at (1,-3) {};
\node [pointnode] (3) at (2,-2) {};
\node [pointnode] (4) at (0,0) {};
\node [pointnode] (5) at (-1,-1) {};
\draw (1) -- (2) -- (3) -- (4) -- (5) -- (1);
\draw (1) -- ++(-1.5,-1.5);
\draw (2) -- ++(0,-1.5);
\draw (3) -- ++(1.5,0);
\draw (4) -- ++(0,1.5);
\draw (5) -- ++(-1.5,0);
\node at (0,-2) {\footnotesize$v_5 + X + Y$};
\node at (2,0) {\footnotesize $v_8$};
\node at (2,-3) {\footnotesize$v_{34} + 2Y$};
\node at (-1,0.3) {\footnotesize$v_{34} + 2X$};
\node at (-2.3,-2) {\footnotesize $v_{12} + 2X + Y$};
\node at (-0.3,-4) {\footnotesize $v_{12} + X + 2Y$};
\end{tikzpicture}
}
\caption{The dual of $\mathcal{S}$.}\label{tropical_curve_pentagon}
\end{subfigure}
\caption{The subdivision and the tropical curve of case (h).}
\end{figure}
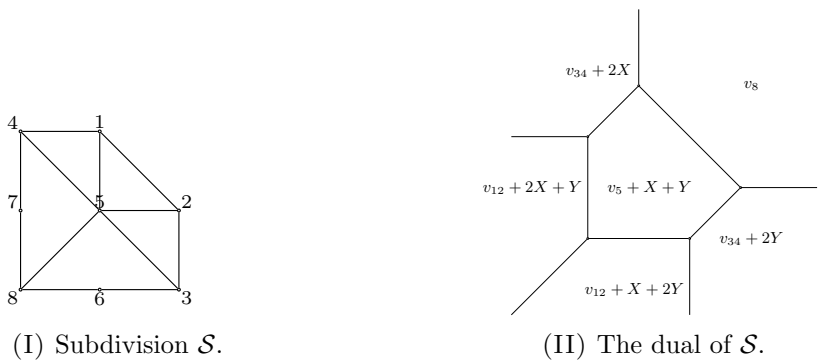
In this case, we have triangle cells that are not primitive. However, notice that a triangle cell is dual to a vertex that separates three areas in the tropical curve. Specifically, the non-smooth tropical curve of case (h) is shown in Figure \ref{tropical_curve_pentagon}. Thus, we always name the cells of a subdivision according to the points of $\Delta_f$ that form vertices of the cell and we have
\begin{equation*}
\mathcal{S} = \{[1,2,5], [1,4,5], [2,3,5], [4,5,8], [3,5,8]\}.
\end{equation*}
\indent Cell $[1,2,5]$ is defined by inequalities (\ref{125}). Cells $[1,4,5]$ and $[2,3,5]$ are defined by inequalities
\begin{align}\label{245}
\begin{split}
-v_{12} + v_{34} + v_5 - v_{67} <& 0 \\
-2v_{12} + v_{34} + 2v_5 - v_8 <& 0
\end{split}
\quad
\begin{split}
-v_{34} + v_5 <& 0 \\
-v_{12} + 2v_5 - v_{67} <& 0.
\end{split}
\end{align}
Lastly, cells [4,5,8] and [3,5,8] are defined by inequalities
\begin{align}\label{458}
\begin{split}
-v_{34} + 2v_5 - 2v_{67} + v_8 <& 0 \\
-2v_{12} - v_{34} + 4v_5 - v_8 <& 0 \\
-2v_{12} + v_{34} + 2v_5 - v_8 <& 0.
\end{split}
\quad
\begin{split}
-v_{34} + v_5 <& 0 \\
v_{34} - 2v_{67} + v_8 <& 0 \\
&
\end{split}
\end{align}
Inequalities (\ref{125}), (\ref{245}), and (\ref{458}) form a polyhedral cone that can be represented by
\begin{equation}\label{polyhedral_pentagon}
-v_{34} + v_5 < 0 \qquad  
-2v_{12} + v_{34} + 2v_5 - v_8 < 0 \qquad  
v_{34} -2v_{67} + v_8 < 0.
\end{equation}
\indent Next, we have to determine the extreme rays of cone (\ref{polyhedral_pentagon}). We will do this by evaluating the following three polyhedral cones with lower dimensions.
\begin{equation}\label{pentagon1}
-v_{34} + v_5 = 0 \qquad  
-2v_{12} + v_{34} + 2v_5 - v_8 < 0 \qquad  
v_{34} -2v_{67} + v_8 < 0 
\end{equation}
\begin{equation}\label{pentagon2}
-v_{34} + v_5 < 0 \qquad  
-2v_{12} + v_{34} + 2v_5 - v_8 = 0 \qquad  
v_{34} -2v_{67} + v_8 < 0
\end{equation}
\begin{equation}\label{pentagon3}
-v_{34} + v_5 < 0 \qquad  
-2v_{12} + v_{34} + 2v_5 - v_8 < 0 \qquad  
v_{34} -2v_{67} + v_8 = 0 
\end{equation}
\indent Cones (\ref{pentagon1}) and (\ref{pentagon2}) have coordinates $v = (1,-1,-1,1,-1)$ and $v = (-2,0,-3,2,-2)$, respectively, that correspond to tropical curves differ from Figure \ref{tropical_curve_pentagon}. Meanwhile, let
$$
v_8 = -v_{34} + 2v_{67}
$$
as mentioned in cone (\ref{pentagon3}) and substitute to (\ref{eq-our-trop-f-sect3}) to have
\begin{align*}
\trop(f)(X, Y) = \min (& v_{12} + X + 2Y, v_{12} + 2X + Y, v_{34} + 2X, v_{34} + 2Y, v_5 + X + Y, \\
& v_{67} + X, v_{67} + Y, -v_{34} + 2v_{67}).
\end{align*}
The relations between each cell, its dual coordinate $(X,Y)$, and the value of $\trop(f)(X,Y)$ are as shown in Table \ref{tab-XYcoord-pentagon}.
\begin{table}[ht]
\centering
\begin{tabular}{lll}
Cell & Corresponding $(X,Y)$ & Terms of $\trop(f)(X,Y)$ with minimum value \\ \hline
$[1,2,5]$ & $(-v_{12}+v_{5},-v_{12}+v_{5})$ & $1^{\text{st}}, 2^{\text{nd}}, 5^{\text{th}}$ \\ \hline
$[1,4,5]$ & $(-v_{12}+v_{34},-v_{12}+v_{5})$ & $1^{\text{st}}, 3^{\text{rd}}, 5^{\text{th}}$ \\ \hline
$[2,3,5]$ & $(-v_{12}+v_{5},-v_{12}+v_{34})$ & $2^{\text{nd}}, 4^{\text{th}}, 5^{\text{th}}$ \\ \hline
$[4,5,8]$ & $(-v_{5} + v_{67},-v_{34}+v_{67})$ & $3^{\text{rd}}, 5^{\text{th}}, 6^{\text{th}}, 8^{\text{th}}$ \\ \hline
$[3,5,8]$ & $(-v_{34}+v_{67},-v_{5} + v_{67})$ & $4^{\text{th}}, 5^{\text{th}}, 7^{\text{th}}, 8^{\text{th}}$ \\ \hline
\end{tabular}
\caption{Coordinates of Figure \ref{tropical_curve_pentagon}.}\label{tab-XYcoord-pentagon}
\end{table}
\noindent In each cell, we see that the terms of $\trop(f)$ that determine the value of $\trop(f)(X,Y)$ do not exceed the points of $\Delta_f$ that are covered by the cell. Thus, we have the subdivision in case (h) occurs if and only if
\begin{equation*}
-v_{34} + v_5 < 0 \quad  
-2v_{12} + v_{34} + 2v_5 - v_8 < 0 \quad  
v_{34} -2v_{67} + v_8 \leq 0
\end{equation*}
hold.
\end{proof}
\section{Applications}\label{sect-app}
\subsection{Symmetric Honeycomb}\label{sect-sym-hc}
Chan-Sturmfels \cite{chan13} considered a cubic in the form of
$$
g(x,y)  =  a(x^3  + y^3  + 1) + b(x^2y  + x^2  + xy^2  + x + y^2  + y ) + xy,
$$
and showed that $C(\trop(g))$ is a symmetric honeycomb if and only if $\val(a) > 2\val(b) > 0$. Here $C(\trop(g))$ is called in honeycomb form if it contains a trivalent hexagonal cycle. Moreover, a tropical curve in honeycomb form is called {\it symmetric} when the lattice lengths of the six edges of the hexagon are equal, and the lattice lengths of the three bounded edges emerging from the hexagon are also equal. In this subsection, we examine our truncated cubic
\begin{align}\label{eq-our-truncated-f-sect5}
f(x,y) = c_{12} (xy^2 + x^2y) + c_{34} (x^2 + y^2) + c_5 xy + c_{67} (x + y) + c_8
\end{align}
and investigate analogous conditions for $C(\trop(f))$.
\begin{Corollary}\label{cor-hc}
The tropical curve of $\trop(f)$ is in honeycomb form if and only if
$$
-v_5 + 2 v_{67} - v_8 < 0 \qquad  
-v_{34} + v_5 < 0 \qquad  
-v_{12} + v_{34} + v_5 - v_{67} < 0.
$$
\end{Corollary}
\begin{proof}
Tropical curve $C(\trop(f))$ contains a trivalent hexagonal cycle if and only if its dual is a regular subdivision containing cells $\{[1,2,5],[1,4,5],[2,3,5],[4,5,7],[3,5,6],[5,6,7]\}$. Thus, this is the case (d) of Table \ref{characterization_table}.
\end{proof}
\begin{table}[ht]
\centering
\begin{tabular}{m{0.5cm}m{6cm}m{3.5cm}m{3.5cm}}
& Edges emenating from the hexagon & Subdivision & Tropical curve \\ \hline
(a) & Five rays and one bounded edge &
\resizebox{0.2\textwidth}{!}{ 
\begin{tikzpicture}
\node [pointnode, label={[label distance=-3.5pt]north:\tiny $1$}] (1) at (1,2) {};
\node [pointnode, label={[label distance=-5pt]north east:\tiny $2$}] (2) at (2,1) {};
\node [pointnode, label={[label distance=-5pt]south east:\tiny $3$}] (3) at (2,0) {};
\node [pointnode, label={[label distance=-5pt]north west:\tiny $4$}] (4) at (0,2) {};
\node [pointnode, label={[label distance=-3.5pt]north :\tiny $5$}] (5) at (1,1) {};
\node [pointnode, label={[label distance=-3.5pt]south:\tiny $6$}] (6) at (1,0) {};
\node [pointnode, label={[label distance=-5pt]north west:\tiny $7$}] (7) at (0,1) {};
\node [pointnode, label={[label distance=-5pt]south west:\tiny $8$}] (8) at (0,0) {};
\draw (1) -- node [right] {\tiny$E_1$} (2) -- node [right] {\tiny$E_2$} (3) -- node [below] {\tiny$E_3$} (6) -- node [below] {\tiny$E_4$} (7) -- node [left] {\tiny$E_5$} (4) -- node [above] {\tiny$E_6$} (1);
\draw (5) -- (1);
\draw (5) -- (2);
\draw (5) -- (4);
\draw (5) -- (7);
\draw (5) -- (6);
\draw (5) -- (3);
\draw (6) -- (8) -- (7);
\end{tikzpicture} 
} &
\resizebox{0.2\textwidth}{!}{ 
\begin{tikzpicture}
\node[pointnode] (1) at (0,0) {};
\node[pointnode] (2) at (1,0) {};
\node[pointnode] (3) at (2,1) {};
\node[pointnode] (4) at (2,2) {};
\node[pointnode] (5) at (1,2) {};
\node[pointnode] (6) at (0,1) {};
\draw [white] (-1.2,-1.2) -- (4.2,4.2);
\draw (1) -- (2) -- (3) -- (4) -- (5) -- (6) -- (1);
\draw (1) -- node[left] {$e_1$} ++(-1,-1);
\draw (2) -- node[right] {$e_6$} ++(0,-1);
\draw (3) -- node[below] {$e_5$} ++(2,0);
\draw (5) -- node[right] {$e_3$} ++(0,2);
\draw (6) -- node[above] {$e_2$} ++(-1,0);
\node[pointnode] (7) at (3,3) {};
\draw (4) -- node[right] {$e_4$ (tail)} (7);
\draw (7) -- ++(1,0);
\draw (7) -- ++(0,1);
\end{tikzpicture} 
} \\  \hline
(b) & Six rays &
\resizebox{0.2\textwidth}{!}{ 
\begin{tikzpicture}
\node [pointnode, label={[label distance=-3.5pt]north:\tiny $1$}] (1) at (1,2) {};
\node [pointnode, label={[label distance=-5pt]north east:\tiny $2$}] (2) at (2,1) {};
\node [pointnode, label={[label distance=-5pt]south east:\tiny $3$}] (3) at (2,0) {};
\node [pointnode, label={[label distance=-5pt]north west:\tiny $4$}] (4) at (0,2) {};
\node [pointnode, label={[label distance=-3.5pt]north :\tiny $5$}] (5) at (1,1) {};
\node [pointnode, label={[label distance=-3.5pt]south:\tiny $6$}] (6) at (1,0) {};
\node [pointnode, label={[label distance=-5pt]north west:\tiny $7$}] (7) at (0,1) {};
\draw (1) -- node [right] {\tiny$E_1$} (2) -- node [right] {\tiny$E_2$} (3) -- node [below] {\tiny$E_3$} (6) -- node [below] {\tiny$E_4$} (7) -- node [left] {\tiny$E_5$} (4) -- node [above] {\tiny$E_6$} (1);
\draw (5) -- (1);
\draw (5) -- (2);
\draw (5) -- (4);
\draw (5) -- (7);
\draw (5) -- (6);
\draw (5) -- (3);
\end{tikzpicture} 
} &
\resizebox{0.2\textwidth}{!}{ 
\begin{tikzpicture}
\node[pointnode] (1) at (0,0) {};
\node[pointnode] (2) at (1,0) {};
\node[pointnode] (3) at (2,1) {};
\node[pointnode] (4) at (2,2) {};
\node[pointnode] (5) at (1,2) {};
\node[pointnode] (6) at (0,1) {};
\draw [white] (-1.2,-1.2) -- (3.2,3.2);
\draw (1) -- (2) -- (3) -- (4) -- (5) -- (6) -- (1);
\draw (1) -- node[left] {$e_1$} ++(-1,-1);
\draw (2) -- node[right] {$e_6$} ++(0,-1);
\draw (3) -- node[below] {$e_5$} ++(1,0);
\draw (5) -- node[right] {$e_3$} ++(0,1);
\draw (6) -- node[above] {$e_2$} ++(-1,0);
\draw (4) -- node[right] {$e_4$} ++(1,1);
\end{tikzpicture} 
} \\ \hline
\end{tabular}
\caption{Two types of truncated honeycomb.}
\label{table-trunc-hnycomb}
\end{table}
\begin{Proposition}[Two types of truncated honeycomb]\label{prop-two-types-trunc-hc}
Let $f$ be as defined in \textup{(\ref{eq-our-truncated-f-sect5})}, and suppose the conditions outlined in Corollary \ref{cor-hc} are satisfied by $\trop(f)$. In this case, the six edges emanating from the hexagonal cycle can be classified as either:
\begin{enumerate}[label=(\alph*)]
\item five rays and one bounded edge (called the tail), or
\item six rays,
\end{enumerate}
as illustrated in \textup{Table \ref{table-trunc-hnycomb}}. The cases (a),(b) occur according to whether $c_8 \neq 0$, $c_8=0$, respectively.
\end{Proposition}
\begin{proof}
The six edges emenating from the hexagonal cycle are the duals of edges $E_1, \dots, E_6$ of the subdivisions on Table \ref{table-trunc-hnycomb}. For $i = 1,2,3,5,6$, the dual of edges $E_i$ are the rays $e_i$ since $E_i$ are parts of the border of $\Delta_f$. When $c_8 \neq 0$, the Newton polygon $\Delta_f$ takes the form shown in case (a). In this scenario, edge $E_4$ does not lie on the border of $\Delta_f$, resulting in its dual edge, $e_4$, being a bounded edge. If $c_8 = 0$, $\Delta_f$ exhibits the shape depicted in case (b). In this case, edge $E_4$ is part of the border of $\Delta_f$, causing $e_4$ to form a ray.
\end{proof}
We shall say a truncated honeycomb $C(\trop(f))$ to be {\it quasi-symmetric} if the six sides of the hexagon have the same lattice length. A quasi-symmetric truncated honeycomb is {\it symmetric} (following the definition in \cite{chan13}) if and only if the hexagon has six emanating rays and does not possess a tail, that is of type (b) of Proposition \ref{prop-two-types-trunc-hc}.
\begin{Proposition}[Quasi-symmetric truncated honeycombs]\label{prop-chan-trunc}
Let $f$ be as in \textup{(\ref{eq-our-truncated-f-sect5})} and suppose $C(\trop(f))$ is a truncated honeycomb. Then $C(\trop(f))$ is quasi-symmetric if and only if
$$
2 v_{34} = v_{12} + v_{67} \text{ and } -v_5 + 2v_{67} < v_8.
$$
The lattice length of the hexagon's side is $\vert v_{34} - v_5 \vert$ and the tail is equal to $\vert v_5 - 2 v_{67} + v_8 \vert$. In particular, $C(\trop(f))$ is symmetric
if and only if
$$
2 v_{34} = v_{12} + v_{67} \text{ and } v_8 =\infty.
$$
\end{Proposition}
\begin{proof}
From case (a) of Table \ref{table-trunc-hnycomb}, a truncated honeycomb tropical curve $C(\trop(f))$ is quasi-symmetric if and only if the lattice lengths of the edges dual to $[5,1]$, $[5,4]$, and $[5,7]$ are equal. The lattice length can be determined by the differences of coordinates $X$ or $Y$. From Table \ref{tab-XYcoord-hexagon} in the proof of Theorem \ref{thm-smooth}, we have Table \ref{table-lat-length} that implies the six edges of the hexagon are equal if and only if
$
\vert v_{34} - v_5 \vert = \vert v_{12} - v_{34} - v_5 + v_{67} \vert.
$
From the last two inequalities of Corollary \ref{cor-hc}, we have
$
v_{34} - v_5 = v_{12} - v_{34} - v_5 + v_{67}
$
, thus $2 v_{34} = v_{12} + v_{67}$. Together with the first inequality of Corollary \ref{cor-hc}, the result follows. Thus, the lattice length of the hexagon's side is $\vert v_{34} - v_5 \vert$, while the tail is the dual of edge $[6,7]$ whose lattice length is $\vert v_5 - 2v_{67} + v_8 \vert$.
\begin{table}[ht]
\centering
\setlength{\extrarowheight}{0.05cm}
\begin{tabular}{cc}
Edges of $\Delta_f$ & The lattice length of its duals \\ \hline
$[5,1]$ & $\vert v_{34} - v_5 \vert$ \\[0.05cm] \hline
$[5,4]$ & $\vert v_{12} - v_{34} - v_5 + v_{67} \vert$ \\[0.05cm] \hline
$[5,7]$ & $\vert v_{34} - v_5 \vert$ \\[0.05cm] \hline 
$[6,7]$ & $\vert v_5 - 2v_{67} + v_8 \vert$ \\[0.05cm] \hline 
\end{tabular}
\caption{The lattice length of some edges on $C(\trop(f))$}\label{table-lat-length}
\end{table}
\\ \indent Meanwhile, truncated honeycomb $C(\trop(f))$ is symmetric if and only if the dual of edge $[6,7]$ has infinite lattice length. That is $\vert v_5 - 2v_{67} + v_8 \vert  = v_5 - 2v_{67} + v_8 = \infty$. Hence, $v_5 = \infty$ or $v_8 = \infty$. If $v_5 = \infty$, the edges $[5,i]$, where $i = 1,2,3,4,6,7$, of the regular subdivisions on Table \ref{table-trunc-hnycomb} do not exist. Thus, $v_8 = \infty$.
\end{proof}
\begin{Example}\label{exmpl-chan-hc}
Let $(v_{12}, v_{34}, v_5, v_{67}) = (3, 2, 0, 1)$ such that $v_{12} \neq v_{34} \neq v_{67}$. If $v_8 = 3$,  $C(\trop(f))$ is a quasi-symmetric truncated honeycomb where the hexagon's sides have length 2 and the tail has length 1 as shown on Figure \ref{fig:f-quasi-hc}. If $v_8 = \infty$, tropical curve $C(\trop(f))$ is a symmetric truncated honeycomb as illustrated in Figure \ref{fig:f-sym-hc}.
\begin{figure}[ht]
\begin{subfigure}[b]{7cm}
\centering
\resizebox{3cm}{!}{ 
\begin{tikzpicture}
\node[pointnode] (1) at (0,0) {};
\node[pointnode] (2) at (2,0) {};
\node[pointnode] (3) at (4,2) {};
\node[pointnode] (4) at (4,4) {};
\node[pointnode] (5) at (2,4) {};
\node[pointnode] (6) at (0,2) {};
\draw (1) -- (2) -- (3) -- (4) -- (5) -- (6) -- (1);
\draw (1) -- ++(-1,-1);
\draw (2) -- ++(0,-1);
\draw (3) -- ++(2,0);
\draw (4) -- ++(1,1);
\draw (5,5) -- ++(1,0);
\draw (5,5) -- ++(0,1);
\draw (5) -- ++(0,2);
\draw (6) -- ++(-1,0);
\end{tikzpicture} 
}
\caption{$(v_{12}, v_{34}, v_5, v_{67}, v_8) = (3,2,0,1,3)$}
\label{fig:f-quasi-hc}
\end{subfigure}  
\begin{subfigure}[b]{7cm}
\centering
\resizebox{3cm}{!}{ 
\begin{tikzpicture}
\node[pointnode] (1) at (0,0) {};
\node[pointnode] (2) at (2,0) {};
\node[pointnode] (3) at (4,2) {};
\node[pointnode] (4) at (4,4) {};
\node[pointnode] (5) at (2,4) {};
\node[pointnode] (6) at (0,2) {};
\draw (1) -- (2) -- (3) -- (4) -- (5) -- (6) -- (1);
\draw (1) -- ++(-1,-1);
\draw (2) -- ++(0,-1);
\draw (3) -- ++(2,0);
\draw (4) -- ++(2,2);
\draw (5) -- ++(0,2);
\draw (6) -- ++(-1,0);
\end{tikzpicture} 
}
\caption{$(v_{12}, v_{34}, v_5, v_{67}, v_8) = (3,2,0,1,\infty)$}
\label{fig:f-sym-hc}
\end{subfigure}  
\caption{Quasi-symmetric and symmetric truncated honeycombs in Example \ref{exmpl-chan-hc}.}
\label{fig-f-hc}
\end{figure}
\end{Example}
\subsection{Nobe's one-parameter family $f_k$}\label{sect-nobe-fk}
In \cite{nobe2008}, Nobe studied a certain piecewise linear dynamical system called the ultradiscrete QRT map and found its invariant curve to be identified with the cycle of a tropical elliptic curve. Fix $(v_{12}, v_{34}, c_{67}, v_8) \in \mathbb{R}^4$ and consider a one-parameter family of tropical curves $\{C(\trop(f_k))\}_{k \in \mathbb{R}}$ with
\begin{align*}
\trop(f_k)(X,Y) = \min( & k + X + Y, v_{12} + X + 2Y, v_{12} + 2X + Y, \\
& v_{34} + 2X, v_{34} + 2Y, v_{67} + X, v_{67} + Y, v_8).
\end{align*}
\noindent According to \cite{nobe2008} Lemma 1, there is a one-parameter family of ultradiscrete QRT maps whose invariant curve $I_k$ coincides with the cycle part of $C(\trop(f_k))$ for each $k \in \mathbb{R}$.
\begin{Example}[Nobe, \protect{\cite[Example 1]{nobe2008}}]
Since we are dealing with operations $(+, \min)$ while \cite{nobe2008} works with $(+, \max)$, substitute negative values of Nobe's parameters as follows, (See Remark \ref{max-min}), 
$$
v_{12} = -10 \quad v_{34} = 0 \quad v_{67} = -5 \qquad v_8 = 0.
$$
\noindent The invariant curves $I_k (k \in \mathbb{R})$ are classified into heptagons, pentagons, squares or (degeneration to) a point respectively for
$$
k \in (-\infty, -15), [-15, -10), [-10, -7.5), [-7.5, \infty).
$$
\noindent If $k$ is in $(-\infty, -15), (-15, -10), (-10, -7.5)$, then $C(\trop(f_k))$ is smooth according to Theorem \ref{thm-smooth} (e), (c), (b), respectively. If $k = -15, -10$, and $k \geq -7.5$, then $C(\trop(f_k))$ is non-smooth and the case corresponds to Theorem \ref{thm-non-smooth} (c), (b), (a), respectively.
\end{Example}
\begin{Example}\label{exmpl-nobe-triangle}
Let us present the case $(v_{12}, v_{34}, v_{67}, v_8) = (0,14,4,0)$. According to Theorem \ref{thm-smooth} (e), the cycle part of $C(\trop(f_k))$ forms a heptagon for $k < -10$. Similarly, Theorem \ref{thm-smooth} (c) and Theorem \ref{thm-non-smooth} (c) reveal that the cycle part becomes a pentagon for $-10 \leq k < -4$. Moving further, for the range $-4 \leq k < 0$, Theorem \ref{thm-smooth} (a) and Theorem \ref{thm-non-smooth} (e) indicate that the cycle part takes the shape of a triangle. Notably, Theorem \ref{thm-non-smooth} (d) states that when $k \geq 0$, the tropical curve $C(\trop(f_k))$ does not contain any cycle. See Figure \ref{fig-f-k}.
\begin{figure}[ht]
\centering
\resizebox{0.3\textwidth}{!}{
\begin{tikzpicture}
\node[pointnode] (0) at (0,0) {};
\node[pointnode] (1) at (-5,10) {};
\node[pointnode] (2) at (10,-5) {};
\node[pointnode] (3) at (-10,15) {};
\node[pointnode] (4) at (15,-10) {};
\draw [dashed] (0) -- ++(-15,-15);
\draw [dashed] (3) -- (1) -- (0) -- (2) -- (4);
\draw [dashed] (3) -- ++(0,5);
\draw [dashed] (3) -- ++(-5,0);
\draw [dashed] (4) -- ++(0,-5);
\draw [dashed] (4) -- ++(5,0);
\draw [dashed] (1) -- ++(0,10);
\draw [dashed] (2) -- ++(10,0);
\draw [line width=3pt] (-2.5,5) -- (5,-2.5) -- node [below] {\HUGE$f_{-1}$} (-2.5,-2.5) -- cycle;
\draw [line width=3pt] (-7.5,12.5) -- (-5,12.5) -- (12.5,-5) -- (12.5,-7.5) -- node [below] {\HUGE$f_{-5}$} (-7.5,-7.5) -- cycle;
\draw [line width=3pt] (-12.5,15) -- (-10,17.5) -- (-5,17.5) -- (17.5,-5) -- (17.5,-10) -- (15,-12.5) -- node [below] {\HUGE$f_{-11}$} (-12.5,-12.5) -- cycle;
\end{tikzpicture}
}
\caption{$C(\trop(f_k))$ for $k = -1, -5, -11$ in Example \ref{exmpl-nobe-triangle}.}
\label{fig-f-k}
\end{figure}
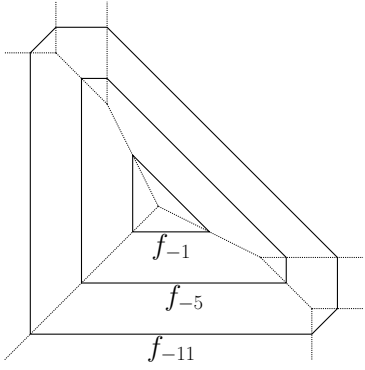
\end{Example}
\subsection{Unimodular transformation}\label{sect-unimod}
Let
$$
g(x,y) = x^2 y^2 \cdot f \left( \frac{1}{x}, \frac{1}{y} \right)
$$
be the result of an integral unimodular transformation on $f$ and we have
$$
g(x,y) = c_{12} (x + y) + c_{34} (x^2 + y^2) + c_5 xy + c_{67} (x^2 y + y^2 x) + c_8 x^2 y^2.
$$
The tropicalization of $g$ is
\begin{align*}
\trop(g)(X,Y) = \min( & v_{12} + X, v_{12} + Y, v_{34} + 2X, v_{34} + 2Y, v_5 + X + Y, \\
& v_{67} + 2X + Y, v_{67} + X + 2Y, v_8 + 2X + 2Y).
\end{align*}
In this section we will show that the conditions of Theorem \ref{thm-smooth} and Theorem \ref{thm-non-smooth} are invariant after such transformation. 
\begin{Lemma}\label{lemma-uni}
Let $f$ and $g$ be the Laurent polynomials that are defined above. Then 
$$
C(\trop(f)) = -1 \cdot C(\trop(g))
$$
holds for the same collection of coefficients $(v_{ij})$.
\end{Lemma}
\begin{proof}
From the tropicalization of
$$
g(x,y) = f \left( \frac{1}{x}, \frac{1}{y} \right) \cdot x^2 y^2,
$$
we have
$$
\trop(g)(X,Y) = \trop(f)(-X, -Y) + 2X + 2Y.
$$
Since $2X + 2Y$ does not exhibit any singularities for all $(X,Y)$, we have $(X,Y)$ is a point on $C(\trop(g))$ if and only if $(-X, -Y)$ is a point on $C(\trop(f))$.
In other words, 
$$
C(\trop(g))(X,Y) = -1 \cdot C(\trop(f))(X,Y)
$$
holds.
\end{proof}
\subsection{Two-parameter family of Edwards curves $f_{r,s}$}\label{sect-edwards}
Let $\mathbb{K}$ be a valuated field, $q \in \mathbb{K}$ with positive valuation, and
$$
\al=\prod_{n=1}^\infty (1+q^n) \quad \text{ and } \quad  
\bal=\prod_{n=1}^\infty (1+(-q)^n).
$$
In \cite{nak23}, we demonstrated that the polynomial
$$
f_{r,s}(x,y) = d_{12} (x + y) + d_{34} (x^2 + y^2) + d_5 xy + d_{67} (x^2 y + y^2 x) + d_8 x^2 y^2
$$
is birationally equivalent to the Edwards curve for certain values of $d_{12}, d_{34}, d_5, d_{67}, d_8 \in \mathbb{K}$. This equivalence allows us to gain insight into the tropicalization of $f_{r,s}$ through a theta function parametrization, akin to the approach employed by Kajiwara-Kaneko-Nobe-Tsuda \cite{kajiwara2009}. Let $u_{k} = \val(d_{k})$ for $k \in \{12, 34, 5, 67, 8\}$. The tropicalization of $f_{r,s}$ is the tropical polynomial
\begin{align*}
\trop(f_{r,s})(X,Y) = \min( & u_{12} + X, u_{12} + Y, u_{34} + 2X, u_{34} + 2Y, u_5 + X + Y, \\
& u_{67} + 2X + Y, u_{67} + 2Y + X, u_8 + 2X + 2Y).
\end{align*}
The coefficients $d_{k}$ are parametrized by two parameters $r,s \in \mathbb{K}$ in the following way.
\begin{align}\label{coef_d_ij}
\begin{split}
d_{12}&=
2\al\bal(\al^4-\bal^4)(\bal s-\al r), \\
d_{34}&= 
(\al^4-\bal^4)(\bal^2s^2-\al^2r^2), \\
d_5&= 
8\al\bal(\al r-\bal s)(\bal^3 r-\al^3 s), \\
d_{67} &= 
2(\al r-\bal s)
\{ (\bal^4-\al^4)rs +2\al\bal (\bal^2r^2-\al^2s^2)\}, \\
d_8 &=
2 (\al^2 s^2-\bal^2 r^2)(\bal^2 s^2- \al^2 r^2).
\end{split}
\end{align}
\begin{figure}[!ht]
\begin{subfigure}[b]{3cm}
\centering
\resizebox{3cm}{!}{
\begin{tikzpicture}
\node [pointnode] (1) at (0,0) {};
\draw (1) -- ++(2,2);
\draw (1) -- ++(2,0);
\draw (1) -- ++(0,2);
\draw (1) -- ++(-2,0);
\draw (1) -- ++(0,-2);
\end{tikzpicture}
}
\caption{}
\label{fig:non-smth-no-edge}
\end{subfigure}
\begin{subfigure}[b]{3cm}
\centering
\resizebox{3cm}{!}{
\begin{tikzpicture}
\node [pointnode] (1) at (-0.5,-0.5) {};
\node [pointnode] (2) at (1,1) {};
\draw (1) -- (2);
\draw (1) -- ++(-1.5,0);
\draw (1) -- ++(0,-1.5);
\draw (2) -- ++(1.5,0);
\draw (2) -- ++(0,1.5);
\draw (2) -- ++(1.5,1.5);
\draw (2) -- ++(-3,0);
\draw (2) -- ++(0,-3);
\end{tikzpicture}
}
\caption{}\label{fig:non-smth-one-edge}
\end{subfigure}
\begin{subfigure}[b]{3cm}
\centering
\resizebox{3cm}{!}{
\begin{tikzpicture}
\node [pointnode] (1) at (0,0) {};
\node [pointnode] (2) at (2,0) {};
\node [pointnode] (3) at (2,2) {};
\node [pointnode] (4) at (0,2) {};
\draw (1) -- (2) -- (3) -- (4) -- (1);
\draw (1) -- ++(-1.5,0);
\draw (1) -- ++(0,-1.5);
\draw (2) -- ++(1.5,0);
\draw (2) -- ++(0,-1.5);
\draw (3) -- ++(1.5,1.5);
\draw (4) -- ++(-1.5,0);
\draw (4) -- ++(0,1.5);
\end{tikzpicture}
}
\caption{}
\label{fig:non-smth-square-only}
\end{subfigure}
\begin{subfigure}[b]{3cm}
\centering
\resizebox{3cm}{!}{
\begin{tikzpicture}
\node [pointnode] (1) at (0,0) {};
\node [pointnode] (2) at (1.5,0) {};
\node [pointnode] (3) at (1.5,1.5) {};
\node [pointnode] (4) at (0,1.5) {};
\node [pointnode] (5) at (-0.5,-0.5) {};
\draw (1) -- (2) -- (3) -- (4) -- (1);
\draw (1) -- (5);
\draw (5) -- ++(-1.5,0);
\draw (5) -- ++(0,-1.5);
\draw (2) -- ++(1,0);
\draw (2) -- ++(0,-2);
\draw (3) -- ++(1,1);
\draw (4) -- ++(-2,0);
\draw (4) -- ++(0,1);
\end{tikzpicture}
}
\caption{}
\label{fig:non-smth-square-one-edge}
\end{subfigure}
\begin{subfigure}[b]{3cm}
\centering
\resizebox{3cm}{!}{
\begin{tikzpicture}
\node [pointnode] (1) at (3,1) {};
\node [pointnode] (2) at (3,-1) {};
\node [pointnode] (3) at (2,-2) {};
\node [pointnode] (4) at (0,0) {};
\node [pointnode] (5) at (1,1) {};
\draw (1) -- (2) -- (3) -- (4) -- (5) -- (1);
\draw (1) -- ++(1.5,1.5);
\draw (2) -- ++(1.5,0);
\draw (3) -- ++(0,-1.5);
\draw (4) -- ++(-1.5,0);
\draw (5) -- ++(0,1.5);
\end{tikzpicture}
}
\caption{}
\label{fig:non-smth-trivalent-pntgn}
\end{subfigure}
\begin{subfigure}[b]{3cm}
\centering
\resizebox{3cm}{!}{
\begin{tikzpicture}
\node [pointnode] (1) at (0,0) {};
\node [pointnode] (2) at (1,0) {};
\node [pointnode] (3) at (1,1) {};
\node [pointnode] (4) at (0,1) {};
\node [pointnode] (5) at (-0.5,-0.5) {};
\node [pointnode] (6) at (1.5,-0.5) {};
\node [pointnode] (7) at (-0.5,1.5) {};
\draw (1) -- (2) -- (3) -- (4) -- (1);
\draw (1) -- (5);
\draw (2) -- (6);
\draw (4) -- (7);
\draw (5) -- ++(-1,0);
\draw (5) -- ++(0,-1);
\draw (6) -- ++(0.5,0);
\draw (6) -- ++(0,-1);
\draw (3) -- ++(1,1);
\draw (7) -- ++(-1,0);
\draw (7) -- ++(0,0.5);
\end{tikzpicture}
}
\caption{}
\label{fig:smth-square}
\end{subfigure}
\begin{subfigure}[b]{3cm}
\centering
\resizebox{3cm}{!}{
\begin{tikzpicture}
\node [pointnode] (1) at (0,0) {};
\node [pointnode] (2) at (0,-2) {};
\node [pointnode] (3) at (-1,-3) {};
\node [pointnode] (4) at (-2,-3) {};
\node [pointnode] (5) at (-3,-2) {};
\node [pointnode] (6) at (-3,-1) {};
\node [pointnode] (7) at (-2,0) {};
\draw (1) -- (2) -- (3) -- (4) -- (5) -- (6) -- (7) -- (1);
\draw (1) -- ++(1,1);
\draw (2) -- ++(1,0);
\draw (3) -- ++(0,-1);
\draw (4) -- ++(0,-1);
\draw (5) -- ++(-1,0);
\draw (6) -- ++(-1,0);
\draw (7) -- ++(0,1);
\end{tikzpicture}
}
\caption{}
\label{fig:smth-heptagon}
\end{subfigure}
\caption{All possible tropical curves of $\trop(f_{r,s})$.}
\label{fig-f-rs}
\end{figure}
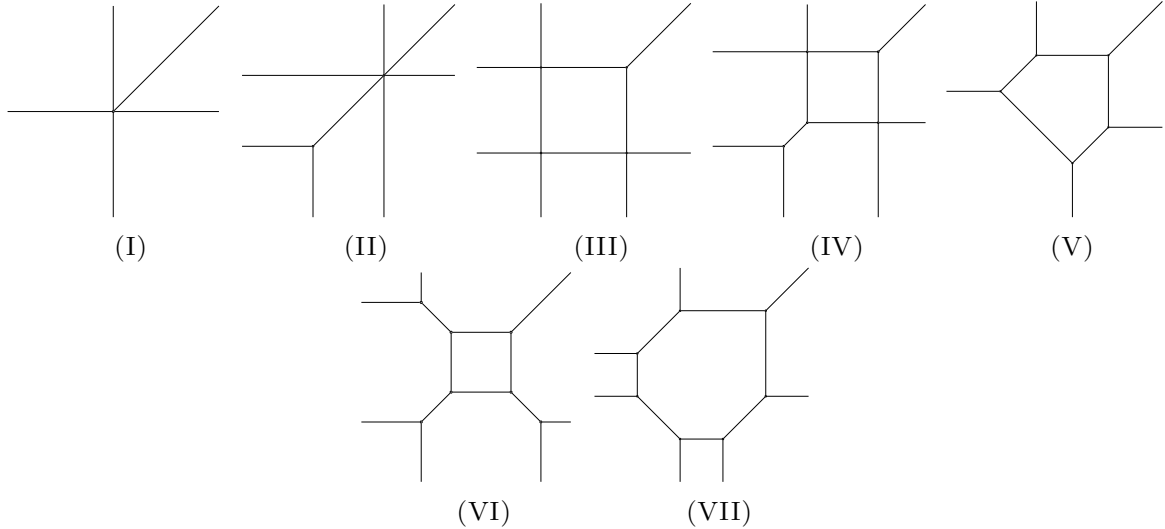
The values of $u_{k} = \val(d_k)$ $(k \in \{12, 34, 5, 67, 8\})$ are contingent not only upon the valuations of $r$ and $s$ but also on their individual coefficients in $q$. The unimodular transformation discussed in $\S$\ref{sect-unimod} allows us to investigate $C(\trop(f_{r,s}))$ in our framework using the truncated symmetric
\begin{align*}
h_{r,s} (x,y) \coloneqq & x^2 y^2 f_{r,s} (x^{-1}, y^{-1}) \\
= & d_{12} (xy^2 + x^2y) + d_{34} (x^2 + y^2) + d_5 xy + d_{67} (x + y) + d_8
\end{align*}
and hence to apply Theorem \ref{thm-smooth} and Theorem \ref{thm-non-smooth}. It turns out that the shapes of 
\begin{align}
C(\trop(f_{r,s})) = -1 \cdot C(\trop(h_{r,s}))
\end{align}
are classified into forms listed in Figure \ref{fig-f-rs}. In particular, the cases \subref{fig:non-smth-square-only}, \subref{fig:non-smth-square-one-edge}, \subref{fig:non-smth-trivalent-pntgn}, \subref{fig:smth-square}, \subref{fig:smth-heptagon} having nontrivial cycles correspond to Thoerem \ref{thm-non-smooth} (f), (g), (h) and Theorem \ref{thm-smooth} (b) (e), respectively. We refer the readers to \protect{\cite[\S 5]{nak23}} for our subsequent discussions.

\end{document}